\documentclass[]{amcjoucc}

\usepackage{amssymb}
\usepackage{hyperref}
\usepackage[capitalize]{cleveref}
\usepackage{multicol} \multicolsep=\smallskipamount
\usepackage{relsize}

\newcommand{\MR}[1]{\href{http://www.ams.org/mathscinet-getitem?mr=#1}{MR#1}}

\newcommand{\csee}[1]{\textup(see \cref{#1}\textup)}
\newcommand{\fullcsee}[2]{\textup(see \fullcref{#1}{#2}\textup)}
 \renewcommand{\Cref}{\cref}
 \newcommand{\pref}[1]{(\ref{#1})}
 \newcommand{\fullcref}[2]{\cref{#1}\pref{#1-#2}}
 \newcommand{\fullCref}[2]{\Cref{#1}\pref{#1-#2}}

\swapnumbers

\theoremstyle{plain}
\newtheorem{thm}[equation]{Theorem}
\newtheorem{prop}[equation]{Proposition}
\newtheorem{cor}[equation]{Corollary}
\newtheorem{lem}[equation]{Lemma}

\newtheorem{FGL}[equation]{Lemma}
\crefformat{FGL}{the Factor Group Lemma~(#2#1#3)}

\crefname{cor}{Corollary}{Corollaries}
\Crefname{cor}{Corollary}{Corollaries}
\crefname{lem}{Lemma}{Lemmas}
\Crefname{lem}{Lemma}{Lemmas}
\crefname{prop}{proposition}{propositions}
\Crefname{prop}{Proposition}{Propositions}

\crefformat{prop}{Proposition~#2#1#3}

\theoremstyle{definition}
\newtheorem{attack}[equation]{Method of attack}
\newtheorem{convention}[equation]{Convention}
\newtheorem{notation}[equation]{Notation}
\newtheorem{defn}[equation]{Definition}

\newtheorem{assump}[equation]{Assumption}

\crefmultiformat{notation}{Notations~#2#1#3}{ and~#2#1#3}{, #2#1#3}{, and~#2#1#3}

\theoremstyle{definition}
\newtheorem{rem}[equation]{Remark}
\newtheorem{rems}[equation]{Remarks}

\crefformat{rems}{Remark~#2#1#3}
 
\newenvironment{justification}{\begingroup \begin{proof}}{\end{proof}\endgroup}

\DeclareMathOperator{\Cay}{Cay}
\DeclareMathOperator{\norm}{Norm}
\DeclareMathOperator{\volt}{volt}

\newcommand{\iso}{\cong}
\newcommand{\quot}{\overline}

\newcommand{\GAP}{\filename{GAP}}
\newcommand{\LKH}{\filename{LKH}}

\newcommand{\CC}{\mathbb{C}}

\newcommand{\ZZ}{\mathbb{Z}}

\newcommand{\normal}{\triangleleft}
\newcommand{\normaleq}{\trianglelefteq}
\newcommand{\notnormal}{\mathrel{\mkern-3mu\not\mkern-3mu\normal}}

\newcommand{\filename}[1]{\texttt{#1}}
\newcommand{\function}[1]{\textsf{#1}}

\makeatletter
\newcommand{\noprelistbreak}{\smallskip\@nobreaktrue\nopagebreak} 
\makeatother

\newcommand{\maynewline}{\hfil\ \penalty1000\hfilneg\ }

\makeatother

\usepackage{ifxetex}
\ifxetex
    \usepackage{fontspec}
\usepackage{mtpro2} % use mathtime pro fonts
\fi	
\makeatletter

\renewcommand{\subsection}{%
\@startsection{subsection}{2}{0pt}{-\baselineskip}{0.5\baselineskip}{\mathversion{bold}\bfseries}%
}

\def\@journalname{}
\def\@journalyear{}
\def\@oddrunninghead{Dave Witte Morris and Kirsten Wilk}
\def\@evenrunninghead{Cayley graphs of order $k p$ are hamiltonian for $k < 48$}

\renewcommand{\ps@amctitle}{%
    \renewcommand\@oddhead{}
    \let\@evenhead\@oddhead
    \renewcommand\@evenfoot{}
    \let\@oddfoot\@evenfoot
  }

\renewcommand{\authordatanoaffil}[3]{%
\advance\@authorcount by 1%
\xdef\@emails{\@emails\ifnum\@authorcount>1 , \fi #2 (#1)}%
\xdef\@authors{\@authors\ifnum\@authorcount>1 , \fi {#1}}%
\if!#3!
{\large #1}% dwm 4-2018 
\else
\global\advance\@footcount by 1%
{\large #1} \symbolfootnote[\@footcount]{#3}%
\fi}

\renewcommand{\authordatatwo}[7]{
\begin{center}
\authordatanoaffil{#1}{#2}{#3} and% dwm 4-2018
\authordatanoaffil{#4}{#5}{#6}\\[1.5mm]
\affiliation{#7}
\end{center}}

\renewenvironment{abstract}
{
\hrule height 0.25pt
\vskip 5pt
\noindent \textbf{Abstract}
\vskip 5pt
}
{
\vskip 5pt
\noindent \textit{\small Keywords:~\@keywords}

\vskip 3pt
\noindent \textit{\small Math.\ Subj.\ Class.: \@msc}
\vskip 5pt
\hrule height 0.25pt
} 

\renewcommand{\section}{%
\@startsection{section}{1}{0pt}{-\baselineskip}{0.5\baselineskip}{\large\bfseries}%
}

\renewenvironment{frontmatter}
{\thispagestyle{amctitle}}
{\vskip 20pt%
\blfootnote{\raggedright \ifnum\@authorcount=1\textit{E-mail address:}\else\textit{E-mail addresses:}\fi ~\@emails\\ \hskip\parindent\textit{URL:} \href{http://people.uleth.ca/~dave.morris/}{http://people.uleth.ca/{\smaller$\sim$}\hskip0.5pt dave.morris/}
\hfill \textsf{preprint version of April 2018}% !!!
}}

\renewcommand{\titledata}[2]{%
\global\def\p@pertitle{#1}%
\vskip 5mm
\begin{center}\LARGE\bfseries\mathversion{bold}% dwm 4-2018 mathversion
#1\if!#2!\else
\global\advance\@footcount by 1%
\symbolfootnote[\@footcount]{#2}\fi\end{center}%
\vskip 1mm
}

\makeatother

\begin{document}

\larger % not for publication !!!

\begin{frontmatter}

\titledata{Cayley graphs of order $k p$ are hamiltonian for $k < 48$}
{}% footnote on the title -- empty if not required

%%%%%%% Two authors from the same institution 
\authordatatwo{Dave Witte Morris}{Dave.Morris@uleth.ca}{}%footnote for first author
{Kirsten Wilk}{kirsten.wilk@uleth.ca}{}%footnote for first author
{Department of Mathematics and Computer Science, University of Lethbridge, Lethbridge, Alberta, T1K\,6R4, Canada}

\keywords{Cayley graph, hamiltonian cycle, hamiltonian connected, hamiltonian laceable}
\msc{05C25, 05C45}

\begin{abstract}
We provide a computer-assisted proof that if $G$ is any finite group of order $kp$, where $1 \le k < 48$ and $p$~is prime, then every connected Cayley graph on~$G$ is hamiltonian (unless $kp = 2$).
As part of the proof, it is verified that every connected Cayley graph of order less than~48 is either hamiltonian connected or hamiltonian laceable (or has valence $\le 2$). 
\end{abstract}

\end{frontmatter}

\section{Introduction}

In a series of papers \cite{M2Slovenian-LowOrder, CurranMorris2-16p, GhaderpourMorris-27p, GhaderpourMorris-30p}, it was shown that if $1 \le k < 32$ (with $k \neq 24$) and $p$~is any prime number, then every connected Cayley graph on every group of order\/~$k p$ has a hamiltonian cycle (unless $kp = 2$). This note extends that work, by treating the previously excluded case $k = 24$, and by increasing the upper bound on~$k$:

\begin{thm} \label{kp}
If\/ $1 \le k < 48$, and $p$~is any prime number, then every connected Cayley graph on every group of order~$k p$ has a hamiltonian cycle \textup(unless $kp = 2$\textup).
\end{thm}

All of the results in the previous papers \cite{M2Slovenian-LowOrder, CurranMorris2-16p, GhaderpourMorris-27p, GhaderpourMorris-30p} were verified  by hand. However, some of the proofs are quite lengthy, so many details were probably never checked by anyone other than the authors and the referees. The present paper takes the opposite approach: many of the results have not been verified by hand, but all of the source code is available in 
	$$ \text{$\langle$\,\emph{the ancillary files directory of this paper on the arxiv}\,$\rangle$} $$
so the results can easily be reproduced by anyone with a standard installation of the computer algebra system \GAP\ \cite{GAP} (including the Small Groups package \cite{SmallGroups}) and G.\,Helsgaun's implementation \LKH\ \cite{LKH} of the Lin-Kernighan heuristic for the traveling salesperson problem.
An effort was made to keep the algorithms in this paper simple, so they would be easy to verify, even though this precluded many optimizations.

In addition to extending the above-mentioned results for $k < 32$, the present work also provides an independent verification of those results, because the proofs are essentially self-contained (other than relying heavily on the correctness of extensive \GAP\ computations).
We also establish the following two results of independent interest:

\begin{cor} \label{ham<144}
If\/ $|G| < 144$ \textup(and $|G| > 2$\textup), then every connected Cayley graph on~$G$ is hamiltonian.
\end{cor}

\begin{prop}  \label{hamconn<64}
If\/ $|G| < 48$, then every connected Cayley graph on~$G$ is either hamiltonian connected or hamiltonian laceable \textup(or has valence $\le 2$\textup).
\end{prop}

\begin{rems} \ 
\noprelistbreak
	\begin{enumerate}
	
	\item The definition of the terms ``hamiltonian connected'' and ``hamiltonian laceable'' can be found in \cref{HamConnLaceDefn}.
	
	\item Almost all of this paper is devoted to the proof of \cref{kp}. \Cref{ham<144,hamconn<64} are proved in \cref{HamConnLaceSect}.
	
	\item It is explained in \cref{ReduceCases} that the paper's calculations for the proof of \cref{kp} could be substantially shortened by accepting all of the results in the literature, rather than reproving some of them. For example, instead of treating all values of~$k$ from~$1$ to~$47$, it would suffice to consider only $k \in \{24, 32, 36, 40, 42, 45\}$ \fullcsee{CanAlsoAssume}{k}
	
	\item It is natural to ask whether the conclusion of \cref{hamconn<64} holds for all Cayley graphs, without any restriction on the order (cf.\ \cite[Questions 4.1 and~4.3, pp.~121--122]{DupuisWagon-Laceable}). This is known to be true when $G$ is abelian \cite{ChenQuimpo} and for a few other (very restricted) classes of Cayley graphs \cite{Alspach-CoxeterI, AlspachChenMcAveney-HamLace, AlspachQin-HamConn, Araki-HyperHamLace}, but \cref{hamconn<64} seems to be the first exhaustive examination of this topic for Cayley graphs of small order.
	Further calculations reported that the conclusion of \cref{hamconn<64} holds for all orders less than $108$, but the additional computations took several weeks and were marred by crashes and other issues, so they are not definitive.

	\end{enumerate}
\end{rems}

\begin{attack}
For each fixed~$k$ and prime number~$p$, there are only finitely many groups~$G$ of order~$kp$ (up to isomorphism), and each of these groups has only finitely many Cayley graphs. Assuming that $kp$ is not too large, \LKH\ can find a hamiltonian cycle in all of them. This means that (given sufficient time) a computer can deal with any finite number of primes. 

Therefore, large primes are the main concern. For these, we have the helpful observation that if $G$ is a group of order~$kp$, where $p$~is prime and $p > k$, then $G$ has a unique Sylow $p$-subgroup (so the Sylow $p$-subgroup is normal), and the Sylow $p$-subgroup is \textup(isomorphic to\textup)~$\ZZ_p$. This means that, after some computer calculations to eliminate the small cases \csee{anomalousSect}, 
	$$ \text{we may assume $\ZZ_p \normal G$, and $p \nmid k$}. $$
 For convenience, 
 	$$ \text{let $\quot G = G/\ZZ_p$, so $|\quot G| = k$.} $$
Since $\ZZ_p$ is cyclic, we are in position to apply the \cref{FGL}: it suffices to find a hamiltonian cycle in $\Cay\bigl( \quot G; \quot S \bigr)$ whose voltage generates~$\ZZ_p$.

There are infinitely many primes~$p$, so a given group~$\quot G$ of order~$k$ is the quotient of infinitely many different groups~$G$. In order to deal simultaneously with all primes, first note that the Schur-Zassenhaus Theorem \cite{Wikipedia-SchurZassenhaus} tells us that $G$ is a semidirect product: $G = \ZZ_p \rtimes_\tau \quot G$ \fullcsee{CanAssume}{semiprod}. We construct a single ``universal'' (infinite) semidirect product $\widetilde G = Z \rtimes_{\widetilde\tau} \quot G$ that has every $\ZZ_p \rtimes_\tau \quot G$ as a quotient. (For example, if all values of the twist homomorphism~$\tau$ are $\pm1$, then $\widetilde G = \ZZ \rtimes_\tau \quot G$.) 

In almost all cases, a computer search yields a hamiltonian cycle $H$ in $\Cay\bigl( \quot G; \quot S \bigr)$, such that its voltage~$\widetilde v$ in $Z \rtimes_{\widetilde\tau} \quot G$ is nonzero. Then $H$ has nontrivial voltage in $\ZZ_p \rtimes_\tau \quot G$ unless $p$ is one of the finitely many prime divisors of~$\widetilde v$.  \LKH\ can verify that all of the (finitely many) Cayley graphs corresponding to these primes are hamiltonian. 
Fortunately, theoretical arguments can handle the few situations where the computer search was unable to find any hamiltonian cycles with nonzero voltage \csee{IrredSpecialCases}.
\end{attack}

\tableofcontents

\section{Preliminaries}

\begin{notation} \label{GSNotation} \ 
	\begin{enumerate}
	\item $G$ is always a group of order~$kp$, where $p$~is prime,
	\item $S$ is a generating set of~$G$,
	and
	\item $\Cay(G;S)$ is the \emph{Cayley graph} on~$G$ with respect to the generators~$S$. The vertices of this graph are the elements of~$G$, and there is an edge joining $g$ and~$sg$ whenever $g \in G$ and $s \in S$.
	\end{enumerate}
\end{notation}

\begin{rem}
Unlike most authors, we do not require $S$ to be symmetric (i.e., closed under inverses). Instead, in our notation, $\Cay(G;S) = \Cay(G;S \cup S^{-1})$.
\end{rem}

Hamiltonian cycles in a subgraph are also hamiltonian cycles in the ambient graph, so, in order to prove \cref{kp}, there is no harm in making the following assumption:

\begin{assump} \label{AssumeIrredundant}
The generating set~$S$ of~$G$ is \emph{irredundant}, in the sense that no proper subset of~$S$ generates~$G$.
\end{assump}

As mentioned in the introduction, the paper relies heavily on the computer algebra system \GAP\ \cite{GAP} and G.\,Helsgaun's implementation \LKH\ of the Lin-Kernighan heuristic \cite{LKH}.

\subsection{\GAP}
The Small Groups library in \GAP\ contains all of the groups of order less than 1024, and many others \cite{SmallGroups}. The number of groups of order~$k$ is given by the function
	\function{NumberSmallGroups(k)}, 
and each group of order~$k$ has a unique id number (from 1 to \function{NumberSmallGroups(k)}). The \GAP\ function
	$$ \function{SmallGroup(k, id)} $$
constructs the group of order~$k$ with the given id number.

The \GAP\ package \textsf{grape} provides tools for working with graphs. In particular, it defines the function
	$$ \function{CayleyGraph(G,S)} $$
that constructs the Cayley graph of the group~$G$ with respect to the generating set~$S$.

To prove \cref{kp}, we wish to show, for certain groups~$G$, that all of the Cayley graphs $\Cay(G;S)$ are hamiltonian. With \cref{AssumeIrredundant} in mind, we would like to have a list of all of the irredundant generating sets of~$G$. However, there is no need to distinguish between Cayley graphs that are isomorphic, so we consider two generating sets to be equivalent if one can be obtained from the other by applying an automorphism of~$G$. Furthermore, since $\Cay(G;S) = \Cay(G;S \cup S^{-1})$, we also consider two generating sets to be equivalent if one can be obtained from the other by replacing some elements by their inverses. The function
	$$ \function{IrredUndirGenSetsUpToAut(G)} $$
constructs a list of all of the irredundant generating sets of~$G$, up to equivalence. It is defined in the file \filename{UndirectedGeneratingSets.gap} and is adapted from the \function{AllMinimalGeneratingSets} algorithm in the masters thesis of B.\,Fuller \cite[pp.~31--34]{Fuller-CCAThesis}. (Fuller's program does not allow generators to be replaced by their inverses.)

Combining \function{IrredUndirGenSetsUpToAut(G)} with \function{CayleyGraph(G,S)} provides a list of all of the irredundant Cayley graphs on any group~$G$.

\subsection{Finding hamiltonian cycles with \LKH\ and exhaustive search}
G.\,Helsgaun's \cite{LKH} implementation \LKH\ of the Lin-Kernighan heuristic is a very powerful tool for finding hamiltonian cycles, and the function
	$$ \text{\function{LKH(X, AdditionalEdges, RequiredEdges)}} $$
interfaces \GAP\ with this program. (It is defined in the file \filename{LKH.gap}.) Given a graph~$X$ (in \textsf{grape} format), and two lists of edges, the function constructs a graph $X^+$ by adding the edges in \function{AdditionalEdges} to~$X$, and asks \LKH\ to find a hamiltonian cycle in~$X^+$ that contains all of the edges in \function{RequiredEdges}. If $X = \function{CayleyGraph(G,S)}$, then the hamiltonian cycle is returned as a list of elements of~$G$, in the order that they are visited by the cycle.

For example, the function
	$$ \function{IsAllHamiltonianOfTheseOrders(OrdersToCheck)} $$
uses \function{LKH} (together with \function{IrredUndirGenSetsUpToAut(G)} with \function{CayleyGraph(G,S)}) to verify that every Cayley graph of order~$k$ is hamiltonian, for every~$k$ in the list \function{OrdersToCheck}. (It is defined in the file \filename{IsAllHamiltonianOfTheseOrders.gap}.)

\function{LKH} returns a single hamiltonian cycle, but we sometimes want several hamiltonian cycles, in order to find one whose voltage is nonzero. The function 
	$$ \function{HamiltonianCycles(X, RequiredEdges)} $$
finds all of the hamiltonian cycles in~$X$ that contain all of the edges in the list \function{RequiredEdges}. (It is defined in the file \filename{HamiltonianCycles.gap}.) However, the list of all hamiltonian cycles may be unreasonably long (and may take too long to compute), so we instead rely on two functions that provide a fairly short list of hamiltonian cycles that suffice for the task at hand:
	$$ \begin{array}{l}
	\function{SeveralHamCycsInCay(GBar,SBar)} \\
	\function{SeveralHamCycsInRedundantCay(GBar, S0Bar, a)}
	\end{array} $$
(Both of these functions are defined in the file \filename{SeveralHamCycsInCay.gap}.) The first provides a list of hamiltonian cycles in $\Cay(\quot G; \quot S)$, whereas the second provides hamiltonian cycles in $\Cay\bigl(\quot G; \quot{S_0} \cup \{\quot a \} \bigr)$. 

\begin{rem}
In order to verify the correctness of the results in this paper, it is not necessary to verify the correctness of the source code of any of the four functions that provide hamiltonian cycles. This is because the output of these functions is always checked for validity before it is used; the function
	$$ \function{IsHamiltonianCycle(X, H, AdditionalEdges, RequiredEdges)} $$
was written for this purpose. It verifies that $H$ is a hamiltonian cycle in the graph~$X^+$ that is obtained from~$X$ by adding the edges in the list \function{AdditionalEdges}, and also that $H$ contains all of the edges in the list \function{RequiredEdges}. Our convention is that each edge $[u,v]$ in \function{AdditionalEdges} and \function{RequiredEdges} is considered to be directed, unless $[v,u]$ is also in the list, in which case the edge is undirected.
\end{rem}

\subsection{Some Cayley graphs that are hamiltonian connected/laceable} \label{HamConnLaceSect}

\begin{defn}[{}{\cite[Defn.~1.3]{AlspachQin-HamConn}}] \label{HamConnLaceDefn}
Let $X$ be a graph.
	\begin{enumerate}
	\item $X$ is \emph{hamiltonian connected} if $X$ has a hamiltonian path from~$v$ to~$w$, for all vertices $v$ and~$w$, such that $v \neq w$.
	\item $X$ is \emph{hamiltonian laceable} if $X$ is bipartite, and it has a hamiltonian path from~$v$ to~$w$, for all vertices $v$ and~$w$, such that $v$ and~$w$ are not in the same bipartition set.
	\end{enumerate}
\end{defn}

\begin{proof}[Justification of \cref{hamconn<64}]
It is easy to write a \GAP\ program that
	\begin{itemize}
	\item loops through all groups~$G$ of order $< 64$,
	\item loops through all irredundant generating sets~$S_0$ of~$G$,
	and
	\item uses \function{LKH} to verify that $\Cay(G;S_0)$ is hamiltonian connected/laceable if the valence is $\ge 3$.
	\end{itemize}
Cayley graphs are vertex transitive, so, for the last step, it suffices to find a hamiltonian path from the identity element~$e$ to all other elements~$a$ of~$G$ (for hamiltonian connectivity) or to all elements~$a$ of the other bipartition set (for hamiltonian laceability). To find this hamiltonian path, one can ask \function{LKH} to find a hamiltonian cycle in the graph $X \cup \{ea\}$, such that the hamiltonian path contains the edge $ea$. (Note that, by symmetry, there is no need to find hamiltonian paths to both of $a$ and~$a^{-1}$.)

However, this is not sufficient to establish \cref{hamconn<64}.
Any generating set~$S$ of~$G$ contains an irredundant generating set~$S_0$, and it is obvious that:
	\begin{itemize}
	\item If $\Cay(G;S_0)$ is hamiltonian connected, then $\Cay(G;S)$ is hamiltonian connected. 
	\item If $\Cay(G;S_0)$ is hamiltonian laceable, \emph{and} $\Cay(G;S)$ is bipartite, then $\Cay(G;S)$ is hamiltonian laceable. 
	\end{itemize}
But it may be the case that $\Cay(G;S_0)$ is bipartite and $\Cay(G;S)$ is not bipartite. In this situation, the hamiltonian laceability of $\Cay(G;S_0)$ does not imply the required hamiltonian connectivity of $\Cay(G;S)$.

Therefore, in cases where $\Cay(G;S_0)$ is bipartite, the program also needs to verify hamiltonian connectivity for generating sets of the form $S = S_0 \cup \{g\}$, such that $\Cay(G;S)$ is not bipartite. (Such a set $S$ can be called a \emph{nonbipartite extension} of~$S_0$.) Note: we may assume that no proper subset of~$S$ generates~$G$ and gives a nonbipartite Cayley graph. (The hamiltonian connectivity of the Cayley graph of such a subset would imply the hamiltonian connectivity of $\Cay(G;S)$.) Since $\Cay(G;S_0)$ is hamiltonian laceable, we already know there are paths from~$e$ to any vertex in the other bipartition set, so only endpoints~$a$ in the bipartition set of~$e$ need to be considered.

Furthermore, if $\Cay(G;S_0)$ has valence two, then it is (usually) not hamiltonian laceable. Therefore, in this case, the program should verify that $\Cay \bigl( G;  S_0 \cup \{g\} \bigr)$ is hamiltonian connected/laceable for all $g \notin \{e\} \cup S \cup S^{-1}$ (except that we need not consider both $g$ and~$g^{-1}$).

The \GAP\ program in \filename{1-3-HamConnOrLaceable.gap} does all of this.
\end{proof}

When dealing with the case $k = 32$, our proof of \cref{kp} also applies the following known result:

\begin{lem}[\cite{Witte-PrimePower}] \label{ham64}
Every connected Cayley graph of order~$64$ is hamiltonian.
\end{lem}

\begin{justification}
This is a special case of the fact that all Cayley graphs of prime-power order are hamiltonian \fullcsee{KutnarEtal}{primepower}. However, to avoid relying on the literature, one can use the function call
	\function{IsAllHamiltonianOfTheseOrders([64])}
to verify this via a few days of computation. (There are over 14,000 Cayley graphs to consider --- most of the 267 groups of order~$64$ have many irredundant generating sets.)
\end{justification}

\begin{proof}[Proof of \cref{ham<144}]
Assume $|G| < 144$. It is known that every connected Cayley graph on any nontrivial $2$-group is hamiltonian \fullcsee{KutnarEtal}{primepower}, so we may assume that $|G|$ is divisible by some prime $p \ge 3$. Then $|G| = kp$, where $k = |G|/p < 144/3 = 48$, so \cref{kp} applies.

It might be possible to avoid appealing to \fullcref{KutnarEtal}{primepower}, by using \function{LKH} to find hamiltonian cycles in all of the Cayley graphs of order 128, but this would be a massive computation, and we did not carry it out.
\end{proof}

\subsection{\texorpdfstring{Cases where the Sylow $p$-subgroup is not $\ZZ_p$ or is not normal}%
	{Cases where the Sylow p-subgroup is not Zp or is not normal}} \label{anomalousSect}

In all later sections of this paper, we will assume that the Sylow $p$-subgroup of~$G$ is isomorphic to~$\ZZ_p$, and is normal in~$G$. The following \lcnamecref{anomalous} deals with the finitely many groups that do not satisfy this hypothesis.
(See \fullcref{CanAssume}{p>k} for a justification of the assumption that $p$~is the largest prime divisor of~$kp$.)

\begin{prop} \label{anomalous}
Let $P$ be a Sylow $p$-subgroup of~$G$, and assume $|G| = kp$, where $p$~is the largest prime divisor of~$kp$, and $k < 48$.
	\begin{enumerate}
	
	\item \label{anomalous-notZp}
	If $P \not\iso \ZZ_p$, then every connected Cayley graph on~$G$ is hamiltonian.
	
	\item \label{anomalous-notnormal}
	If $P \iso \ZZ_p$ and $P \notnormal G$ , then every connected Cayley graph on~$G$ is hamiltonian.
	\end{enumerate}
\end{prop}

\begin{justification}
\pref{anomalous-notZp}
Since the Sylow $p$-subgroup of~$G$ is not isomorphic to~$\ZZ_p$, we know that $p^2$ is a divisor of $|G| = kp$, so $p \mid k$. In fact, $p$~must be the largest prime divisor of~$k$ (since it is the largest prime divisor of~$kp$). So $p$ is uniquely determined by~$k$. 

It is a simple matter to write a \GAP\ program that
	\begin{itemize}
	\item loops through the values of~$k$ in $\{1,\ldots,47\}$, 
	\item loops through all the nonabelian groups~$G$ of order~$kp$, where $p$~is the largest prime divisor of~$k$, 
	\item loops through all the irredundant generating sets~$S$ of~$G$ (up to automorphisms of~$G$),
	and
	\item uses \function{LKH} to verify that $\Cay(G;S)$ is hamiltonian.
	\end{itemize}
(See the file \filename{2-7(1)-SylowSubgroupNotZp.gap}.) The calculations take several hours to complete. About half of the time is spent finding hamiltonian cycles in the Cayley graphs of order $32 \times 2 = 64$, since there are so many of them, so we separated out that part of the calculation \csee{ham64}. 

One important modification to the algorithm deals with the problem that the original version of the program ran out of memory when trying to find the generating sets of \function{SmallGroup(1058,4)}. (This group arises for $k = 23$.) Since $1058 = 2 \times 23^2$ is of the form $2p^2$, \fullcref{KutnarEtal}{kp2} tells us that every Cayley graph on this group is hamiltonian. (In fact, this group is of ``dihedral type'' so it is very easy either to find all of the irredundant generating sets by hand, or to prove that every connected Cayley graph is hamiltonian.) Therefore, the program skips this group (and prints the comment that it ``\verb|is dihedral type of order 2p^2|\rlap.'')

\medbreak

\pref{anomalous-notnormal} 
Let $d$ be the number of Sylow $p$-subgroups of~$G$. We know from Sylow's Theorem that $d$ is a divisor of~$k$, and that $d \equiv 1 \pmod{p}$. This implies $d < k$ (indeed, $p < d$ since $d \equiv 1 \pmod{p}$, and $d \le k$, since $d$ is a divisor of~$k$). Therefore, for each~$k$, there are only finitely many possibilities for~$p$. Also note that $d > 1$, since the Sylow subgroup~$\ZZ_p$ is not normal, and therefore has conjugates.

It is a simple matter to write a \GAP\ program that
	\begin{itemize}
	\item loops through the values of~$k$ in $\{1,\ldots,47\}$, 
	\item loops through the primes $p$ that are:
		\begin{itemize}
		\item greater than the largest prime divisor of~$k$, 
		\item less than or equal to~$k$, 
		and 
		\item such that there is a divisor~$d$ of~$k$, with $d > 1$ and $d \equiv 1 \pmod{p}$,
		\end{itemize}
	\item loops through all the groups~$G$ of order~$kp$, such that a Sylow $p$-subgroup is not normal,
	\item loops through all the irredundant generating sets~$S$ of~$G$ (up to automorphisms of~$G$),
	and
	\item uses \function{LKH} to verify that $\Cay(G;S)$ is hamiltonian.
	\end{itemize}
(See the file \filename{2-7(2)-SylowSubgroupNotNormal.gap}.) 
\end{justification}

\subsection{Notation and assumptions} \label{AssumpSect}

\begin{notation} \label{GNotation} 
In the remainder of this paper:
\noprelistbreak
	\begin{enumerate}
	\item $G$ is always a group of order~$k p$, where $1 \le k < 48$, and $p$~is a prime number.
	\item $S$ is a generating set of~$G$.
	\item $\quot {\phantom{x}} \colon G \to G/\ZZ_p$ is the natural homomorphism, if it is the case that $\ZZ_p$ is the unique Sylow $p$-subgroup of~$G$. 
	\end{enumerate}
\end{notation}

\begin{convention}
To avoid treating $k = 2$ as a special case, we will consider the graph $K_2$ to be hamiltonian, because it has a closed walk that visits all the vertices exactly once before returning to the starting point.
\end{convention}

\begin{notation}
For $s_1,\ldots,s_n \in S \cup S^{-1}$, we use $(s_1,\ldots,s_n)$ to denote the walk in $\Cay(G;S)$ that visits (in order), the vertices
	$$ e, \, s_1, \ s_1 s_2, \, s_1 s_2 s_3, \ \ldots, \, s_1s_2\cdots s_n .$$
\end{notation}

\begin{defn}[cf.\ {\cite[\S2.1.3, p.~61]{GrossTucker-TopGraphThy}}]
For any hamiltonian cycle $H = (\quot{s_1},\quot{s_2},\ldots,\quot{s_n})$ in the Cayley graph $\Cay(\quot G; \quot S)$, we let $\volt_{G,S}(H) = \prod_{i=1}^n s_i$ be the \emph{voltage} of~$H$.
This is an element of~$\ZZ_p$.
\end{defn}

We wish to show that $\Cay(G;S)$ has a hamiltonian cycle. 
Our main tool is the following elementary observation:

\begin{FGL}[``Factor Group Lemma'' {\cite[\S2.2]{WitteGallian-survey}}] \label{FGL}
Suppose
 \begin{itemize}
% \item $S$ is a subset of~$G$,
 \item $H = (\quot{s_1}, \quot{s_2}, \ldots, \quot{s_k})$ is a hamiltonian cycle in $\Cay(\quot G; \quot S)$,
 and
 \item $\volt_{G,S}(H)$ generates~$\ZZ_p$.
 \end{itemize}
 Then $(s_1,s_2,\ldots,s_m)^p$ is a hamiltonian cycle in $\Cay(G;S)$.
 \end{FGL}

 \begin{lem} \label{CanAssume}
 To prove \cref{kp}, we may assume:
 	\begin{enumerate}
	
	\item \label{CanAssume-nonabel}
	$G$ is not abelian.
	
	\item \label{CanAssume-k>1}
	$k > 1$.
	
	\item \label{CanAssume-k'}
	If $G'$ is any group of order $k' p'$, where $1 \le k' < k$, and $p'$ is any prime number, then every connected Cayley graph on~$G'$ is hamiltonian.
	
	\item \label{CanAssume-p>k}
	$p$ is strictly greater than the largest prime factor of~$k$.
	
	\item \label{CanAssume-Zp}
	$\ZZ_p$ is a Sylow $p$-subgroup of~$G$, and $\ZZ_p \normal G$. 
	
%	\item \label{CanAssume-irred}
%	The generating set~$S$ is irredundant.
	
	\item \label{CanAssume-sNotNormal}
 	There does not exist $s \in S$, such that $\langle s \rangle
\normaleq G$, and such that either
         \begin{enumerate}
         \item \label{CanAssume-sNotNormal-Z} 
         $s \in Z(G)$,
         or
         \item \label{CanAssume-sNotNormal-notZ} 
         $Z(G) \cap \langle s \rangle = \{e\}$,
         or
         \item \label{CanAssume-sNotNormal-p}
         $|s|$ is prime.
         \end{enumerate}
         
	\item \label{CanAssume-notZp}
	$S \cap \ZZ_p = \emptyset$.
	
	\item \label{CanAssume-DoubleEdge}
	$\quot s \neq \quot t$, for all $s,t \in S \cup S^{-1}$ with $s \neq t$.
	
	\item \label{CanAssume-DoubleOrder2}
	If $s \in S$ with $|\quot s| = 2$, then $|s| = 2$.
	
	\item \label{CanAssume-semiprod}
	$G = \ZZ_p \rtimes_\tau \quot G$, where $\tau$~is a homomorphism from~$\quot G$ to~$ \ZZ_p^\times$.

	\end{enumerate}
\end{lem}

\begin{proof}
\pref{CanAssume-nonabel}
Showing that all connected Cayley graphs on abelian groups are hamiltonian is an easy exercise. (The Chen-Quimpo Theorem \pref{ChenQuimpo} is a much stronger result.).

\pref{CanAssume-k>1}
If $k = 1$, then $|G| = p$, so $G$ is abelian, contrary to~\pref{CanAssume-nonabel}.

\pref{CanAssume-k'} We may assume this by induction on~$k$.

\pref{CanAssume-p>k} Let $p'$ be the largest prime factor of~$k$, and write $|G| = k' p'$. If $p = p'$, then $|G|$ is divisible by~$p^2$, so \fullcref{anomalous}{notZp} applies. If $p < p'$, then $k > k'$, so \pref{CanAssume-k'} applies.

\pref{CanAssume-Zp}
If either $P \not\iso \ZZ_p$ or $P \notnormal G$, then \cref{anomalous} applies.

%\pref{CanAssume-irred}
%Suppose some proper subset~$S_0$ of~$S$ generates~$G$. By choosing $S_0$ of minimal cardinality, we may assume that it is irredundant. Since $\Cay(G; S_0)$ is a spanning subgroup of $\Cay(G;S)$, we know that every hamiltonian cycle in $\Cay(G; S_0)$ is also a hamiltonian cycle in $\Cay(G; S)$.

\pref{CanAssume-sNotNormal}
For any $s \in S$, we know, from \pref{CanAssume-nonabel}, that $\langle s \rangle \neq G$.
We see from \pref{CanAssume-k'} that $\Cay \bigl( G/\langle s \rangle; S \bigr)$ is hamiltonian. Therefore, it is well known (and easy to prove) that if $s$ satisfies any of the given conditions, then $\Cay(G;S)$ is hamiltonian \cite[Lem.~2.27]{M2Slovenian-LowOrder}.

\pref{CanAssume-notZp}
This is a special case of~\pref{CanAssume-sNotNormal-p}.

\pref{CanAssume-DoubleEdge}
From \pref{hamconn<64}, we see that every edge of $\Cay(\quot G; \quot S)$ is in a hamiltonian cycle. Therefore, if $\quot s = \quot t$ with $s \neq t$, then the existence of a hamiltonian cycle in $\Cay(G;S)$ is a well-known (and easy) consequence of \cref{FGL} (cf.~\cite[Cor.~2.11]{M2Slovenian-LowOrder}).

\pref{CanAssume-DoubleOrder2}
Since $\quot s = \quot s^{-1}$, this follows from \pref{CanAssume-DoubleEdge} with $t = s^{-1}$.

\pref{CanAssume-semiprod}
From \pref{CanAssume-p>k}, we know that $\gcd \bigl( |\quot G|, k \bigr) = 1$. Therefore, the desired conclusion is a consequence of the Schur-Zassenhaus Theorem \cite{Wikipedia-SchurZassenhaus}.
\end{proof}

\begin{rem}
It is immediate from \pref{CanAssume-notZp} and \pref{CanAssume-DoubleEdge} of \cref{CanAssume} that the Cayley graphs $\Cay(G;S)$ and $\Cay(\quot G; \quot S)$ have the same valence (and have no loops).
\end{rem}

\section{Irredundant generating sets of the quotient} \label{IrredundantSect}

In this \lcnamecref{IrredundantSect}, we assume that the generating set~$\quot S$ of~$\quot G$ is irredundant.
The assumptions stated in \cref{GSNotation,GNotation,CanAssume} are also assumed to hold. 

In most cases, we will find a hamiltonian cycle in $\Cay(\quot G; \quot S)$ with nonzero voltage, so that \cref{FGL} applies. The following \lcnamecref{IrredSpecialCases} deals with the exceptional cases in which this approach does not work.

\begin{lem} \label{IrredSpecialCases}
Assume the generating set $\quot S$ of~$\quot G$ is irredundant.
Then $\Cay(G;S)$ has a hamiltonian cycle in each of the following situations:
\smallskip
	\begin{enumerate}

	\item \label{IrredSpecialCases-23} 
	$\quot S = \{\quot a, \quot b\}$, with\/ $|\quot a| = 2$, $|\quot b| = 3$, and $\tau(\quot b) = 1$.
	
 	\item \label{IrredSpecialCases-A4xZp}
	$\quot G \iso A_4$,  $\quot S= \{\quot a, \quot b\}$, where $|\quot a| = |\quot b| = 3$, and $\quot G$~centralizes~$\ZZ_p$. 

	\item \label{IrredSpecialCases-GeneralizeA4}
	$\quot G = (\ZZ_2 \times \ZZ_2) \rtimes_\tau \ZZ_m$, $\quot S= \{\quot a, \quot b\}$, where $\quot a \in \ZZ_2 \times \ZZ_2$, $|\quot b| = m$, $\quot G$ is not abelian, and $\quot G$~centralizes~$\ZZ_p$. 

	\item \label{IrredSpecialCases-CentralInS}
	$\quot S$ contains an element~$a$, such that $a \in Z(\quot G)$, $|a| = 2$, and $\tau(a) = 1$.

	\item \label{IrredSpecialCases-a2Normal}
	 $\quot S$ contains an element~$a$, such that $a^2$ has prime order, $\langle a^2\rangle \normal \quot G$, and $\tau(a) = -1$.
	 
	 \item \label{IrredSpecialCases-GisDihedral}
	$\quot G$ is a dihedral group of order~$k$ \textup(with $k > 4$\textup), $\quot S = \{\quot a, \quot b\}$ with $|\quot a| = 2$ and $\quot b = k/2$, $\quot a$ inverts~$\ZZ_p$ and $\quot b$ centralizes~$\ZZ_p$.
	
	\item \label{IrredSpecialCases-ZqxZ4}
	$|\quot G| = 4q$ and $|[\quot G, \quot G]| = q$, where $q$ is prime, and $\quot S = \{\quot a, \quot b\}$, where $|\quot a| = 4$ and $|\quot b| = 2$. Furthermore, $\quot G$~centralizes~$\ZZ_p$, but $\quot b$ does not centralize~$[\quot G, \quot G]$.
	
	\end{enumerate}
\end{lem}

\begin{proof}
\pref{IrredSpecialCases-23}
We may assume $a$ projects trivially to~$\ZZ_p$. (If $a$ centralizes~$\ZZ_p$, this follows from \fullcref{CanAssume}{DoubleOrder2}. If $a$ does not centralize~$\ZZ_p$, then it is true after conjugation by some element of~$\ZZ_p$.) So $b$ must project nontrivially. Since $b$ centralizes~$\ZZ_p$, this implies $|b| = 3p$.

Since $|\quot a| = 2$ and $|\quot b| = 3$, it is easy to see that every hamiltonian cycle in $\Cay(A_4; \quot a, \quot b)$ is of the form
	$ (\quot a, {\quot b\,}^{\pm2}, \quot a, {\quot b\,}^{\pm2}, \ldots, \quot a, {\quot b\,}^{\pm2} ) $
\cite[p.~238]{Schupp-structure}. Hence, each right coset of $\langle \quot b \rangle$ appears as consecutive vertices in this cycle, so it is not difficult to see that 
	$ (  a, b^{\pm(3p-1)}, \quot a, b^{\pm(3p-1)}, \ldots, a, b^{\pm(3p-1)} ) $
passes through all of the vertices in each right coset of $\langle b \rangle$, and is therefore a hamiltonian cycle in $\Cay(G; a,b)$. See Subcase~1.1 of \cite[\S3]{Morris-nonsolvable} for a detailed verification of a very similar example.

\pref{IrredSpecialCases-A4xZp}
\cite[Subcase~2.2 of Prop.~7.2]{M2Slovenian-LowOrder}:
Assume, without loss of generality, that $a$ projects nontrivially to~$\ZZ_p$, so $|a| = 3p$. Therefore $4  |a| = |G|$.
Since $\quot G$~centralizes~$\ZZ_p$, we have $G \iso \ZZ_p \times A_4$. Therefore, $[G,G] \iso [A_4, A_4] \iso \ZZ_2 \times \ZZ_2$, so $| [a,b] | = 2$. 
%So $2 |a| \, | [a,b]| = |G|$. 
It is now not difficult to verify that 
	$\bigl( a^{|a|-1}, b^{-1}, a^{-(|a|-1)}, b \bigr)^2$
is a hamiltonian cycle. (This is a special case of a lemma of D.\,Jungreis and E.\,Friedman that can be found in \cite[2.14]{M2Slovenian-LowOrder}.)

\pref{IrredSpecialCases-A4xZp}
 \fullCref{CanAssume}{DoubleOrder2} tells us that the projection of~$a$ to~$\ZZ_p$ is trivial. So the projection of $b$ to~$\ZZ_p$ is nontrivial. Since $b$ centralizes~$\ZZ_p$, this implies $|b| = mp = |G|/4$. Also note that $\quot b$ does not centralize $\quot a$ (since $\quot G$ is not abelian), so $\{e, b^{-1}a b, b^{-1}a b \, a, a\} = \ZZ_2 \times \ZZ_2$. Therefore, it is easy to see that
 	$ (b^{mp-1}, a, b^{-(mp-1)}, a)^2$
is a hamiltonian cycle in $\Cay(G;S)$. (This is an easy special case of the same lemma of D.\,Jungreis and E.\,Friedman that was used in \pref{IrredSpecialCases-GeneralizeA4}.) 

\pref{IrredSpecialCases-CentralInS} Let $s \in S$ with $\quot s = a$. \fullCref{CanAssume}{DoubleEdge} (with $t = a^{-1}$) implies $|s| = 2$.  Since $\tau(a) = 1$, we know that $s$ centralizes~$\ZZ_p$, so \fullcref{CanAssume}{DoubleOrder2} implies that $s$ has trivial projection to~$\ZZ_p$ (since $p > 2$). Therefore, we have $s \in Z(G)$, which contradicts \fullcref{CanAssume}{sNotNormal-Z}.

\pref{IrredSpecialCases-a2Normal} Let $s \in S$ with $\quot s = a$. Since $\tau(a) \neq 1$, we know that $s$ does not centralize~$\ZZ_p$, so may assume (after conjugating by an appropriate element of~$\ZZ_p$) that the projection of~$s$ to~$\ZZ_p$ is trivial. This means $s = a$. Then, since $\tau(a^2) = \bigl(\tau(a)\bigr)^2 = (-1)^2 = 1$, we see that $s^2$ generates a subgroup of prime order that is normal in~$G$. (Indeed, we know, by assumption, that $\langle a^2\rangle$ is normalized by~$\quot G$, and it centralizes~$\ZZ_p$ since $\tau(s^2) = 1$).  This contradicts the conclusion of \fullcref{CanAssume}{DoubleEdge} (with $t = s^{-1}$, and with $\langle s^2 \rangle$ in the role of~$\ZZ_p$).

\pref{IrredSpecialCases-GisDihedral} Since $\quot a$ does not centralize~$\ZZ_p$, we may assume that $a$~projects trivially to~$\ZZ_p$ (after replacing $S$ with a conjugate). Since $S$ generates~$G$, this implies that $b$~projects nontrivially to~$\ZZ_p$. Since $\quot b$ centralizes~$\ZZ_p$, we conclude that $|b| = p |\quot b| \, p$. Also, since $a = \quot a$ inverts both $\quot b$ and~$\ZZ_p$, we know that $a$ inverts~$b$. So $G$ is the dihedral group of order~$kp$, and $\{a,b\}$ is the obvious generating set consisting of a reflection~$a$ and a rotation~$b$. Therefore, if we let $m = \frac{1}{2}|G| - 1$, then we have the hamiltonian cycle $(a, b^m)^2$.

\pref{IrredSpecialCases-ZqxZ4} This is a known result. Namely, since $|G| = 4pq$, this is a special case of \cref{KutnarEtal}{kpq}.  (Alternatively, we may apply \fullcref{G'=pq}{primepower}, since $[G,G] = |[\quot G, \quot G]| = q$.) For completeness, we record a proof that is adapted from \cite[Case 5.3]{KeatingWitte}.)
We know that $|\quot G| = 4q$, $|[\quot G, \quot G]| = q$, and $|\quot a| = 4$, so we may write $\quot G = \ZZ_q \rtimes \ZZ_4$, with $\ZZ_q = [\quot G, \quot G]$ and $\ZZ_4 = \langle \quot a \rangle$.
Since $\quot b$ has order~$2$ and centralizes~$\ZZ_p$, we see from \fullcref{CanAssume}{DoubleOrder2} that $b$ projects trivially to~$\ZZ_p$, so $a$ must project nontrivially. Therefore $a$ generates $G/\ZZ_q$, so we have $b \in a^i \ZZ_q$, for some (even)~$i$ with $0 \le i < 4p$. (Also, we know $i \neq 0$, because $|\quot b| = 2$ is not a divisor of~$q$.)  Then $(b, a^{-(i-1)}, b, a^{4q-i-1})$ is a hamiltonian cycle in $\Cay(G/\ZZ_q; a,b)$. 

If we write $b = \gamma a^i$, with $\gamma \in \ZZ_q$, then the voltage of this hamiltonian cycle is
	$$ b  a^{-(i-1)}  b a^{4q-i-1} 
	= (\gamma a^i)  a^{-(i-1)}  (\gamma a^i) a^{4q-i-1} 
	= \gamma a  \gamma a^{-1}
	. $$
Since $b$ does not centralize~$\ZZ_q$ (and $b \in a^i \ZZ_q$ with $i$~even), we know that $a$ does not invert~$\ZZ_q$. Therefore the voltage $\gamma a  \gamma a^{-1}$ is nontrivial, so \cref{FGL} applies.
\end{proof}

We wish to show, for each lift of~$\quot S$ to a generating set~$S$ of~$G$, that some hamiltonian cycle in $\Cay(\quot G; \quot S)$ has nonzero voltage.

\begin{defn}[\cite{Wikipedia-norm}]
Recall that the \emph{norm} of an algebraic number is the product of all of its Galois conjugates in~$\CC$.
\end{defn}

\begin{lem}[{cf.\, \cite[Lem.~2.11]{Morris-nonsolvable}}] \label{VoltageMatrix}
Assume 
\noprelistbreak
	\begin{itemize}
	\item $G = \ZZ_p \rtimes_\tau \quot G$, where $\tau$ is a homomorphism from~$\quot G$ to~$ \ZZ_p^\times$,
	\item $\zeta = \phi \circ \tau$, where $\phi$ is an isomorphism from $\ZZ_p^\times$ onto the group~$\mu_{p-1}$ of $(p-1)$th roots of unity in~$\CC$, so $\zeta$ is an abelian character of~$\quot G$ \textup(more precisely, $\zeta$ is a homomorphism from~$\quot G$ to~$\mu_{p-1}$\textup),
	\item $Z$ is the subring of~$\CC$ that is generated by the $(p-1)$th roots of unity,
	\item $S = \{a_1,a_2,\ldots, a_m, b_1,\ldots,b_n\} \cup B_0$ is a generating set of~$G$,
		such that
		\begin{itemize}
			\item each $\quot{a_i}$ has order~$2$, and centralizes~$\ZZ_p$,
			\item either $B_0$ is empty, or $B_0$ consists of a single element~$b_0$ that does not centralize~$\ZZ_p$,
		\end{itemize}
	\item $H_i$ is a hamiltonian cycle in $\Cay(\quot G; \quot S)$, for $i = 1,2,\ldots,n$,
	and
	\item for $j = 1,2,\ldots,n$, $S_j$ is the generating set of~$G$, such that $\quot{S_j} = \quot S$, and $s \in \quot G$ for all $s \in S_j$, except that $(1, \quot{b_j}) \in S_j$.
	\end{itemize}
If\/ $\norm \bigl( \det \bigl[ \volt_{Z \rtimes_\zeta \quot G, S_j}(H_i) \bigr] \bigr)$ is not divisible by~$p$, then\/ $\Cay(G;S)$ is hamiltonian.
\end{lem}

\begin{proof}
From \fullcref{CanAssume}{DoubleOrder2}, we know that $a_i = (0, \quot{a_i})$ for each~$i$. Also, if $B_0$ has an element~$b$, then we may assume $b = (0, \quot b)$, after conjugating by an element of~$\ZZ_p$. So $b_1,\ldots,b_n$ are the only elements of~$S$ that contribute to $\volt_{\ZZ_p \rtimes_\tau \quot G, S}(H)$. Therefore, if we write $b_j = (z_j, \quot{b_j})$, then, from the definition of $S_1,\ldots,S_n$, we have
	$$ \volt_{\ZZ_p \rtimes_\tau \quot G, S}(H)
		= \sum_{j=1}^n \, z_j   \volt_{\ZZ_p \rtimes_\tau \quot G, S_j}(H) . $$
Note that $z_1,\ldots,z_n$ cannot all be~$0$, since $\langle S \rangle = G$. 
Therefore, if $\volt_{\ZZ_p \rtimes_\zeta \quot G, S}(H_i) = 0$ for all~$i$, then elementary linear algebra tells us that
	\begin{align} \label{det=0} \tag{$*$}
	\text{$\Delta_p = 0$, where $\Delta_p = 
	 \det \bigl[ \volt_{\ZZ_p \rtimes_\tau \quot G, S_j}(H_i) \bigr]$}
	 . \end{align}
We will show that this leads to a contradiction. (So there must be a hamiltonian cycle with nonzero voltage, so \cref{FGL} applies.)

The isomorphism $\phi^{-1}\colon \mu_{p-1} \to \ZZ_p^\times$ extends to a unique ring homomorphism $\Phi \colon Z \to \ZZ_p$. Since $\Phi \circ \zeta = \tau$ (and $\Phi$ is a ring homomorphism), it is easy to see that pairing $\Phi$ with the identity map on~$\quot G$ yields a group homomorphism $\widehat \Phi \colon Z \rtimes_\zeta \quot G \to \ZZ_p \rtimes_\tau \quot G$. Therefore
	$$ \Phi \bigl( \volt_{Z \rtimes_\zeta \quot G, S_j}(H) \bigr) = \volt_{\ZZ_p \rtimes_\tau \quot G, S_j}(H) $$
for every hamiltonian cycle~$H$ in $\Cay(\quot G; \quot S)$. Since $\Phi$ is a ring homomorphism (and determinants are calculated simply by adding and multiplying), this implies
	$$ \text{$\Phi(\Delta) = \Delta_p$, where $\Delta = \det \bigl[ \volt_{Z \rtimes_\zeta \quot G, S_j}(H_i) \bigr]$} .$$

The assumption that $\norm (\Delta)$ is not divisible by~$p$ tells us that $\Phi \bigl( \norm (\Delta) \bigr)\neq 0$. Since, by definition, $\norm (\Delta)$ is the product of~$\Delta$ with its other conjugates, and the ring homomorphism~$\Phi$ respects multiplication, we conclude that $\Phi(\Delta) \neq 0$. In other words, $\Delta_p \neq 0$. This contradiction to \pref{det=0} completes the proof.
\end{proof}

\begin{prop} \label{ApplyVoltageMatrix}
If the generating set~$\quot S$ of~$\quot G$ is irredundant, then $\Cay(G;S)$ is hamiltonian.
\end{prop}

\begin{justification} 
For each group~$\quot G$ of order less than 48, and each irredundant generating set~$\quot S$ of~$\quot G$, the \GAP\ program in the file \filename{3-4-IrredundantSBar.gap} constructs a list \function{SeveralHamCycsInCG} of some hamiltonian cycles in $\Cay(\quot G; \quot S)$ (by calling the function \function{SeveralHamCycsInCay}). 

Now, the program considers each abelian character $\zeta$ of~$\quot G$. 
If \cref{IrredSpecialCases} (or some other lemma) provides a hamiltonian cycle in $\Cay \bigl( \ZZ_p \rtimes_\tau \quot G; S \bigr)$, then nothing more needs to be done. Otherwise, the program constructs the list $S_1,\ldots,S_n$ of generating sets described in \cref{VoltageMatrix}, and calculates the voltage $\volt_{Z \rtimes_\zeta \quot G, S_j}(H_i)$ for each $H_i$ in \function{SeveralHamCycsInCG}.

Now, the program calls the function \function{FindNonzeroDet}, which returns a list $i_1,\ldots,i_n$ of indices. The program then verifies that if we use $H_{i_1},\ldots,H_{i_n}$ as the hamiltonian cycles in \cref{VoltageMatrix}, then the norm of the determinant of the voltages  is nonzero. Hence, \cref{VoltageMatrix} provides a hamiltonian cycle in $G = \ZZ_p \rtimes_\tau \quot G$ for all but the finitely many primes~$p$ that are a divisor of this norm.

To deal with these remaining primes, the program calls the function \function{CallLKHOnLiftsOfSBar}, which constructs every possible lift of~$\quot S$ to a generating set~$S$ of~$G$, and uses \function{LKH} to verify that $\Cay(G;S)$ is hamiltonian.
\end{justification}

\begin{rem}
It is not necessary to verify the source code of \function{SeveralHamCycsInCay} or \function{FindNonzeroDet}, because the output of both of these programs is validated before it is used.
\end{rem}

\section{Redundant generating sets of the quotient} 

We now assume that the generating set~$\quot S$ of~$\quot G$ is redundant (but $S$ is irredundant, and the other assumptions stated in \cref{GSNotation,GNotation,CanAssume} are also assumed to hold). The following well-known observation tells us that every $S$ of this type can be constructed by choosing an irredundant generating set~$\quot{S_0}$ of~$\quot G$ and an element~$\quot a$ of~$\quot G$, and letting $S = \bigl( \{0\} \times \quot{S_0} \bigr) \cup \{(1,\quot a)\}$.

\begin{lem} \label{S=S0+a}
Assume the generating set $\quot S$ of~$\quot G$ is redundant. Then, perhaps after conjugating by an element of~$\ZZ_p$, there is an element~$a$ of~$\quot S$, such that if we let $S_0 = S \smallsetminus \{a\}$, then 
	\begin{enumerate}
	\item $\quot {S_0}$ is an irredundant generating set of $\quot G$, 
	and
	\item $S_0 \subseteq \{0\} \rtimes \quot G$.
	\end{enumerate}
\end{lem}

\begin{proof}
By assumption, there is a proper subset~$S_0$ of~$S$, such that $\langle \quot{S_0} \rangle = \quot G$. 
By choosing $S_0$ to be of minimal cardinality, we may assume that $\quot {S_0}$ is irredundant. Since $|\langle S_0 \rangle|$ is divisible by $|\langle \quot{S_0} \rangle| = |\quot G| = |G|/p$, and is a proper divisor of $|G|$, we must have $|\langle S_0 \rangle| = |G|/p$. So $\langle S_0 \rangle$ is a maximal subgroup of~$G$. Therefore, we have $\langle S_0, a \rangle = G$ for any element~$a$ of~$S$ that is not in~$S_0$. Since $S$ is irredundant, we conclude that $S = S_0 \cup \{a\}$.

Since $|\langle S_0 \rangle| = |G|/p$, we see from \fullcref{CanAssume}{p>k} that $\langle S_0 \rangle$ is a Hall subgroup of~$G$. Then, since $\ZZ_p$ is a solvable normal complement, the Schur-Zassenhaus Theorem \cite{Wikipedia-SchurZassenhaus} tells us that, after passing to a conjugate, we have $\langle S_0 \rangle = \{0\} \rtimes \quot G$.
\end{proof}

\begin{lem} \label{RedundantByHamConn}
Assume 
	\begin{itemize}
	\item $S = \bigl( \{ 0 \}\times \quot{S_0} \bigr) \cup \{ (1,\quot a) \}$, 
	\item $\quot{S_0}$ is an irredundant generating set of~$\quot G$, 
	\item either $\Cay(\quot G; \quot{S_0})$ is not bipartite, or $\Cay(\quot G; \quot{S})$ is  bipartite,
	and
	\item $|S_0 \cup S_0^{-1}| \ge 3$.
	\end{itemize}
Then $\Cay(G;S)$ is hamiltonian.
\end{lem}

\begin{proof}
We know from \fullcref{CanAssume}{notZp} that $\quot a \neq \quot e$. Therefore, \cref{hamconn<64} tells us there is a hamiltonian path $(\quot{s_i})_{i=1}^{n-1}$ from~$\quot e$ to~$\quot a \, ^{-1}$ in $\Cay(\quot G; \quot{S_0})$. So $H = \bigl( \quot a, (\quot{s_i})_{i=1}^{n-1} \bigr)$ is a hamiltonian cycle in $\Cay(\quot G; \quot S)$.

Write $a = (z, \quot a)$, with $z \in \ZZ_p \smallsetminus \{0\}$. 
Since $S_0  \subseteq \{0\} \rtimes \quot G$, we must have $z \neq 0$, and the voltage $a s_1s_2\cdots s_{n-1}$ of~$H$ is~$z$. Hence, \cref{FGL} provides a hamiltonian cycle in $\Cay(G ; S)$.
\end{proof}

To complete the proof of \cref{kp}, the following two results consider the special cases that are not covered by \cref{RedundantByHamConn}.

\begin{prop} \label{BipAndNonbip}
Assume 
	\begin{itemize}
	\item $S = \bigl( \{ 0 \}\times \quot{S_0} \bigr) \cup \{ (1,\quot a) \}$, 
	\item $\quot{S_0}$ is an irredundant generating set of~$\quot G$, 
	\item $\Cay(\quot G; \quot{S_0})$ is bipartite,
	and 
	\item $\Cay(\quot G; \quot{S})$ is not bipartite.
	\end{itemize}
Then $\Cay(G;S)$ is hamiltonian.
\end{prop}

\begin{justification}
The \GAP\ program in \filename{4-3-RedundantSBar.gap}:
	\begin{itemize}
	\item loops through all groups~$\quot G$ of order less than $48$,
	\item loops through all irredundant generating sets~$\quot{S_0}$ of~$\quot G$, such that $\Cay(\quot G; \quot{S_0})$ is bipartite,
	\item loops through all nonidentity elements~$\quot a$ of~$\quot G$, such that $\Cay(\quot G; \quot S)$ is not bipartite, where $\quot S = \quot{S_0} \cup \{\quot a \}$,
	\item constructs the set $S = \bigl( \{0\} \times \quot{S_0} \bigr) \cup \{(1, \quot a)\}$,
	\item makes a list of a few hamiltonian cycles in $\Cay(\quot G; \quot S)$ (by calling the function \function{SeveralHamCycsInRedundantCay},
	\item loops through all abelian characters $\zeta$ of~$\quot G$,
	\item ignores this character if the condition in \fullcref{CanAssume}{DoubleOrder2} is not violated,
	\item ignores this character if $S$ is not a minimal generating set of~$G$,
	\item calculates the GCD of the norms of the voltages of the hamiltonian cycles in the list,
	and
	\item uses \function{LKH} to find a hamiltonian cycle in $\Cay(\ZZ_p \rtimes_\tau \quot G; S)$ for each prime~$p$ that divides the GCD, by calling \function{CallLKHOnLiftsOfSBar}. 
	\end{itemize}
(The use of \function{CallLKHOnLiftsOfSBar} in the last step is overkill, because we are interested only in the one particular lift $S$ of~$\quot S$, but we are calling a function that checks all possible lifts. It does not seem worthwhile to write an verify another \GAP\ program, just to eliminate this slight waste.)	
\end{justification}

\begin{rem}
It is not necessary to verify the source code of \function{SeveralHamCycsInRedundantCay}, because the output of this program is validated before it is used.
\end{rem}

\begin{lem} \label{valence2}
Assume 
	\begin{itemize}
	\item $S = \bigl( \{ 0 \}\times \quot{S_0} \bigr) \cup \{ (1,\quot a) \}$, 
	\item $\quot{S_0}$ is an irredundant generating set of~$\quot G$, 
	and 
	\item $|S_0 \cup S_0^{-1}| \le 2$.
	\end{itemize}
Then $\Cay(G;S)$ is hamiltonian.
\end{lem}

\begin{justification}
Since $\quot{S_0}$ is a generating set of~$\quot G$, and $\quot a \notin \{\quot e\} \cup \quot{S_0} \cup \quot{S_0}^{-1}$ (by~\pref{CanAssume-notZp} and~\pref{CanAssume-DoubleEdge} of \cref{CanAssume}), it is easy to see that we must have $k \ge 4$. Also note that the only groups with a $2$-valent, connected Cayley graph are cyclic groups and dihedral groups, and that the $2$-valent generating set of such a group is unique, up to an automorphism of the group.

Applying the same method as in \cref{BipAndNonbip}, the \GAP\ program in \filename{4-5-Valence2.gap}:
	\begin{itemize}
	\item loops through all values of~$k$ from~$4$ to $47$,
	\item loops through the groups~$\quot G$ of order~$k$ that have a $2$-valent, connected Cayley graph, and defines $\quot{S_0}$ to be the $2$-valent generating set of~$\quot G$,
	\item loops through all nonidentity elements~$\quot a$ of~$\quot G$, such that $\quot a \notin \{\quot e\} \cup \quot{S_0} \cup \quot{S_0}^{-1}$ (except that we do not need to consider both $\quot a$ and~$\quot a\,^{-1}$),
	\item constructs the generating set $S = \bigl( \{0\} \times \quot{S_0} \bigr) \cup \{(1, \quot a)\}$ of~$G$,
	\item makes a list of $20$ hamiltonian cycles in $\Cay(\quot G; \quot S)$,
	\item loops through all abelian characters $\zeta$ of~$\quot G$,
	\item ignores this character if the condition in \fullcref{CanAssume}{DoubleOrder2} is not violated,
	\item ignores this character if $S$ is not a minimal generating set of~$G$,
	\item calculates the GCD of the norms of the voltages of the hamiltonian cycles in the list,
	and
	\item uses \function{LKH} to find a hamiltonian cycle in $\Cay(\ZZ_p \rtimes_\tau \quot G; S)$ for each prime~$p$ that divides the GCD, by calling \function{CallLKHOnLiftsOfSBar}. 
	\end{itemize}
(As in \cref{BipAndNonbip}, the use of \function{CallLKHOnLiftsOfSBar} in the last step is overkill.)	
\end{justification}

\section{Known results that can reduce the number of cases} \label{ReduceCases}

There are several results in the literature that can be used to substantially reduce the number of Cayley graphs considered in the proof of \cref{kp} (but then the proof is not self-contained). The following \lcnamecref{KutnarEtal} of Kutnar et al.\ is the main example.

\begin{thm}[{}{\cite[Thm.~1.2]{M2Slovenian-LowOrder}, \cite{Witte-PrimePower}}] \label{KutnarEtal}
Every connected Cayley graph on~$G$ has a
hamiltonian cycle if\/ $|G|$ has any of the following forms\/ {\upshape(}where
$p$, $q$, and~$r$ are distinct primes\/{\upshape):}
\begin{multicols}{2}
 \begin{enumerate}
 
 \item \label{KutnarEtal-kp}
 \newcounter{kp} \setcounter{kp}{\theenumi}
 $k p$, where $1 \le k < 32$, with $k \neq 24$, 
 
 \item \label{KutnarEtal-kpq}
 \newcounter{kpq} \setcounter{kpq}{\theenumi}
 $k p q$, where $1 \le k \le 5$,
 
 \item \label{KutnarEtal-pqr}
 \newcounter{pqr} \setcounter{pqr}{\theenumi}
 $p q r$,
 
 \item \label{KutnarEtal-kp2}
 \newcounter{kpsquared} \setcounter{kpsquared}{\theenumi}
 $k p^2$, where $1 \le k \le 4$, 
 
 \item \label{KutnarEtal-kp3}
 \newcounter{kpcubed} \setcounter{kpcubed}{\theenumi}
 $k p^3$, where $1 \le k \le 2$.
 
\item \label{KutnarEtal-primepower}
 $p^k$.

 \end{enumerate}
 \end{multicols}
 \end{thm}
 
The following result is also useful.
 
 \begin{thm}[\cite{KeatingWitte,Morris-G'=pqOdd,Morris-G'=2p}] \label{G'=pq}
 Every connected Cayley graph on~$G$ has a hamiltonian cycle if either
 	\begin{enumerate}
	\item \label{G'=pq-primepower}
	$[G,G]$ is cyclic of prime-power order, 
	or
	\item $|[G,G]| = pq$, where $p$ and~$q$ are distinct primes, and $|G|$ is odd,
	or
	\item $|[G,G]| = 2p$, where $p$ is an odd prime.
	\end{enumerate}
\end{thm}

 \begin{lem} \label{CanAlsoAssume}
 To prove \cref{kp}, one may assume:
 	\begin{enumerate}
	
	\item  \label{CanAlsoAssume-k}
	$k \in \{24, 32, 36, 40, 42, 45\}$.
	
	\item \label{CanAlsoAssume-G'3}
	$|[\quot G, \quot G]| \ge 3$.
	
	\item \label{CanAlsoAssume-G'4abel}
	either $|[\quot G, \quot G]| \ge 4$, or the twist function $\tau$ is nontrivial.

	\end{enumerate}
\end{lem}

\begin{proof}
\pref{CanAlsoAssume-k} 
If $k < 32$ and $k \neq 24$, then \fullcref{KutnarEtal}{kp} applies. Therefore, either $k$ is in the specified set, or $k \in \{33, 34, 35, 37, 38, 39, 41, 43, 44, 46, 47\}$, in which case some part of \cref{KutnarEtal} applies:
\begin{center}
\begin{tabular}{|c|c||c|c|}
\noalign{\hrule}
\lower3pt\hbox{\Huge\vphantom{$($}}
$k$ & form of $|G| = kp$ & $k$ & form of $|G| = kp$\cr
\noalign{\hrule}
	$ 33$ & $9p$ (if $p = 3$), $3p^2$ (if $p = 11$), or $pqr$ & 
		$ 41$ & $p^2$ (if $p = 41$) or $pq$ \cr	
	$ 34$ &$2p^2$ (if $p = 17$) or $2pq$ & 
		$ 43$ & $p^2$ (if $p = 43$) or $pq$ \cr
	$ 35$ & $25p$ (if $p = 5$), $5p^2$ (if $p = 7$), or $pqr$ & 
		$ 44$ & $4p^2$ (if $p = 11$) or $4pq$ \cr
	$ 37$ & $p^2$ (if $p = 37$) or $pq$ & 
		$ 46$ & $2p^2$ (if $p = 23$) or $2pq$ \cr
	$ 38$ & $2p^2$ (if $p = 19$) or $2pq$ & 
		$ 47$ & $p^2$ (if $p = 47$) or $pq$ \cr
	$ 39$ & $9p$ (if $p = 3$), $3p^2$ (if $p = 13$), or $pqr$ & 
		& \cr
\noalign{\hrule}
\end{tabular}
\end{center}

\pref{CanAlsoAssume-G'3}
The commutator subgroup of~$G$ is a subgroup of $\ZZ_p \rtimes_\tau \quot G'$, so its order is a divisor of $p \, |\quot G'|$. Therefore, if $|[\quot G, \quot G]| \le 2$, then $|[G,G]|$ is either $1$, $2$, $p$, or~$2p$. (Furthermore, if $p = 2$, then $\tau$ must be trivial, so $G = \ZZ_2 \times \quot G$, which implies that $[G,G] = [\quot G, \quot G]$.) So \cref{G'=pq} establishes that every connected Cayley graph on~$G$ has a hamiltonian cycle.

\pref{CanAlsoAssume-G'4abel}
As in \pref{CanAlsoAssume-G'3}, if $\tau$ is trivial, then $G = \ZZ_p \times \quot G$, so $[G,G] = [\quot G, \quot G]$. Therefore, \fullcref{G'=pq}{primepower} provides a hamiltonian cycle in every Cayley graph on~$G$ if $|[\quot G, \quot G]|$ is prime. (In particular, if $|[\quot G, \quot G]| < 4$.)
\end{proof}

\begin{rem}
If we apply \fullcref{CanAlsoAssume}{k}, then the proof of \cref{kp} requires hamiltonian connectivity/laceability only for Cayley graphs of the orders listed in \fullcref{CanAlsoAssume}{k}, not the full strength of \cref{hamconn<64}.
\end{rem}

The computations to justify \cref{hamconn<64} could be shortened a bit by applying the following interesting result:

\begin{thm}[{}{Chen-Quimpo \cite{ChenQuimpo}}] \label{ChenQuimpo}
Assume $\Cay(G;S)$ is a connected Cayley graph. If $G$ is abelian, and the valence of\/ $\Cay(G;S)$ is at least three, then $\Cay(G;S)$ is either hamiltonian connected or hamiltonian laceable.
 \end{thm}

\end{document}